\documentclass[]{interact}

\usepackage{graphicx}              
\usepackage{subfigure}
\usepackage{amsmath}               
\usepackage{amsfonts}              
\usepackage{amssymb}
\usepackage{tabularx,caption}

\newcommand{\real}{\mathbb{R}}
\newcommand{\integ}{\mathbb{Z}}
\newcommand{\posint}{\mathbb{Z}_+}
\newcommand{\sF}{\mathcal{F}}
\newcommand{\bE}{{\mathbb E}}

\usepackage[numbers,sort&compress]{natbib}
\bibpunct[, ]{[}{]}{,}{n}{,}{,}

\theoremstyle{plain}
\newtheorem{theorem}{Theorem}[section]
\newtheorem{lemma}[theorem]{Lemma}
\newtheorem{cor}[theorem]{Corollary}

\newtheorem{assume}[theorem]{Property}

\theoremstyle{definition}

\newtheorem{example}[theorem]{Example}

\theoremstyle{remark}
\newtheorem{remark}{Remark}

\begin{document}


\title{On the Validity of the Girsanov Transformation 
 Method for Sensitivity Analysis of Stochastic Chemical Reaction Networks}

\author{
\name{Ting Wang\textsuperscript{a}\thanks{\textsuperscript{a}Email: tingw@udel.edu} and Muruhan Rathinam\textsuperscript{b}\thanks{\textsuperscript{b}Email: muruhan@umbc.edu}}
\affil{\textsuperscript{a}Department of Mathematical Sciences, University of Delaware, Newark, USA; \textsuperscript{b}Department of Mathematics and Statistics, University of Maryland Baltimore County, Baltimore, USA}
}

\maketitle

\begin{abstract}
We investigate the validity of the Girsanov Transformation (GT) method for parametric sensitivity
analysis of stochastic models of chemical reaction networks. The validity depends on
the likelihood ratio process being a martingale and the commutation of a certain
derivative with expectation. We derive some exponential integrability
conditions which imply both these requirements. We provide further conditions
in terms of a reaction network that imply these exponential integrability
conditions.
\end{abstract}


\section{Introduction}
Parametric sensitivity analysis is an essential part of 
modeling and analysis of dynamical systems. In the context of stochastic 
dynamical systems the problem that is considered frequently is that 
of estimating the {\em sensitivity} defined by the partial derivative
\[
\left.\frac{\partial}{\partial c}\right|_{c = c^*} \mathbb{E} f(X(T,c))
\]
where $c$ is a parameter of interest, $c^*$ is its nominal value, 
$X$ is the stochastic process, $f$ is a scalar function of the state 
and $T>0$ is a fixed terminal time.   

While we focus on the well-stirred stochastic model of chemical kinetics \cite{Gillespie77}, we like to mention that 
similar models arise in applications that are concerned with populations (nonnegative
integer vectors). Due to the high dimensionality of the state space of 
the Markov process $X$ describing the chemical kinetics, Monte Carlo methods 
are usually the most viable. Monte Carlo methods of sensitivity analysis 
for stochastic chemical models can be classified into the {\em finite
  difference} methods (FD) \cite{CFD, CRP}, the {\em Girsanov
Transformation} (GT) method \cite{Gir}, the {\em regularized pathwise derivative} (RPD)
method \cite{RPD} and what might be termed the {\em auxiliary path} (AP) type methods
\cite{APA, PPA}. When considering more general applications, one again finds roughly,
a similar classification \cite{Glynn-book}. Among these methods, the FD
methods are always biased and the RPD method is biased in the context of 
chemical kinetics and is not always applicable. The well known GT method 
is usually widely applicable and is unbiased. The main shortcoming of the GT
method is that it has been observed and that it often has large variance and 
hence less efficient \cite{wang2016efficiency}. 
For an asymptotic analysis of the variance of the GT and 
related methods, see \cite{wang2016efficiency} where a centralized GT method 
is shown to be more efficient under certain circumstances. The recently introduced auxiliary path type
methods are unbiased as well, and provide for an alternative to the GT
methods. Nevertheless, due to the ease of implementation, the GT method is
of interest.  

The validity of the GT method has been studied in the literature under various
contexts \cite{glynn1995likelihood, l1995note}. We briefly describe the
differences and similarities of our work in comparison with these earlier
works.

First we note that the analysis in \cite{glynn1995likelihood, l1995note} which provides sufficient conditions
under which the GT method is applicable is formulated in terms of a
discrete time Markov chain. On the other hand, the analysis in our paper is formulated
in terms of a continuous time Markov chain and provides a self-contained description of the GT method based on \cite{Bremaud}. While in principle there is a
precise relationship between the two approaches, this relationship is rather 
cumbersome. In particular, the fact that a fixed final time
$T$ in our continuous time setting translates to a random stopping time $N$
in the discrete time approach results in added assumptions on the nature of
the dependence of $N$ on parameters: see for instance the assumption {\em A2(iii)} in
\cite{glynn1995likelihood} which assumes 
$L^1$-differentiability (with respect to the parameter) of the ratios of the
values of $N$
at a general parameter value to that at a reference parameter value.
There is no such assumption in our analysis.

Secondly, our analysis is focused on a limited, albeit practically important
class of continuous time Markov chains, namely, stochastic reaction networks. 
This imposes extra structure, resulting in a simpler set of sufficient
conditions.  

The previous works \cite{glynn1995likelihood, l1995note} as well as our work,
all result in a form of exponential integrability requirement as part of
the set of sufficient conditions ensuring the validity of the GT method.  
The works \cite{glynn1995likelihood, l1995note} require exponential
integrability of the event count, while our paper provides one of two
alternative requirements: \eqref{eqn:main-integrability-right} or \eqref{eqn:main-integrability-left}. One requires exponential integrability
of event count while the other requires exponential integrability of
the intensity (propensity) function. This second alternative condition
doesn't appear to have a counter part in \cite{glynn1995likelihood,
  l1995note}. Verification of either one of these is
adequate under the assumption that the required sensitivity exists, which was
shown to follow under milder conditions \cite{APA}. 

The verification of either one of these exponential integrability conditions is not so
straightforward. 
Bearing this in mind, in the context of stochastic chemical reaction networks,
we derive some sufficient conditions for the validity of GT method
that can be verified via an algorithm given the network parameters.
We also illustrate via examples how our sufficient conditions may be verified
for a rich class of reaction networks using a probabilistic coupling argument.

\subsection{Stochastic chemical kinetics}

We describe the stochastic model of well-stirred chemical reactions 
involving $n$ molecular species undergoing $m$ reaction channels \cite{Gillespie77}.  
In this model, the molecular copy number vector $X(t) \in \mathbb{Z}_+^n$ ($t \geq 0$) 
is considered as a Markov process in continuous time. 
Occurrence of $j$th reaction leads to a change of $X(t)$ by $\nu_j \in
\integ^n$ for $j=1,\dots,m$ where the $\nu_j$ are known as the {\em
  stoichiometric vectors}. 
We denote by $R_j(t)$ the counting process which counts the number of
occurrences of reaction channel $j$ during $(0,t]$, for $j=1,\dots,m$.
The probabilistic rate of occurrence of reaction $j$, is given by the
{\em intensity function} or {\em propensity function} $a_j(x)$ which is 
defined such that conditioned on $X(t) = x$, the
probability that $X(t+h)=x + \nu_j$ is $a_j(x) h + o(h)$ as $h \to 0+$, 
and the probability that $X(t+h)=x$ is $1 - \sum_{j=1}^m a_j(x) h + o(h)$ as
$h \to 0+$.   

The propensity function, in addition to the state $x$, potentially depends on other
factors such as the temperature and system volume and this dependence is 
captured by a set of parameters which are non-random and constant in time. 
In particular, in the stochastic form of {\em mass action} case, the propensity
function is of the {\em product form} 
\begin{equation}\label{eq-mass-action}
a_j(x,c) = c_j b_j(x),
\end{equation}
where $c_j > 0$ is a parameter independent of $x$ and $b_j(x)$ is a
(multivariate) polynomial
in $x$ \cite{Gillespie77}. 
While our final results in this paper assume the product form
\eqref{eq-mass-action} (but not necessarily the mass action form), we shall keep our derivations as general as possible 
until the final steps.

It is possible to represent the processes for different parameter
values $c$ in the same sample space $(\Omega,\sF,\mathbb{P})$ (hence the 
notation $X(t,c)$ and $R(t,c)$) 
via the {\em random time change representation} \cite{Ethier-Kurtz}
\begin{equation}\label{eqn:random-time-change}
X(t,c) = x_0 + \sum_{j=1}^m \nu_j Y_j\left( \int_0^t a_j(X(s,c),c) ds \right),
\end{equation}  
where $Y_1,\dots,Y_m$ are independent unit rate Poisson processes 
carried by $(\Omega,\sF,\mathbb{P})$ and $x_0 \in \posint^n$ is the initial state. The reaction count processes $R_j$ are 
then given by 
\begin{equation}
R_j(t,c) = Y_j\left( \int_0^t a_j(X(s,c),c) ds \right) \quad j=1,\dots,m.
\end{equation}
We also note the relationship 
\begin{equation}\label{eq-X-R}
X(t,c) = x_0 + \sum_{j=1}^m \nu_j R_j(t,c).
\end{equation}

We note that the processes $X(t,c)$ and $R(t,c)$ are {\em cadlag}. 
We shall also assume that $X(t,c)$ is non-explosive for each $c$, that is,
$R_j(t,c) < \infty$ for each $t \geq 0$, each $c$ and $j=1,\dots,m$.  
For sake of readability, throughout the paper we suppress the dependence of $X$ and $R$ on $\omega \in
\Omega$ except when necessary. 

Without loss of generality we shall focus on estimation of sensitivity 
with respect to one scalar parameter $c$ and we assume that it corresponds to 
the first reaction channel, so that $c=c_1$ and thus $a_1(x,c)=c \, b_1(x)$ under the product form. 
Let $c^* >0$ be a {\em nominal parameter value} and $T>0$ be some terminal time.  
Given a function $f:\posint^n \to \real$, 
we are interested estimating the sensitivity defined by
\[\left.\frac{\partial}{\partial c}\right|_{c = c^*} \mathbb{E} f(X(T,c)).\]
Throughout this paper, we assume that the sensitivity exists.
We refer the reader to \cite{APA} for some sufficient conditions
that guarantee the
existence of the sensitivity. 

\subsection{The Girsanov transformation method}\label{sec:GT-method}
One of the commonly used sensitivity estimation methods is the Girsanov
transformation (GT) method which is also known as the likelihood ratio (LR)
method in literature. 
We first describe the basics of this approach and then furnish details 
in the context of chemical kinetics.  
We consider a nominal parameter value $c^*$ and an open interval 
$I_{c^*} = (c^*-\epsilon,c^*+\epsilon)$. 
As mentioned before, we assume that the processes $X(t,c)$ and $R(t,c)$ for $c \in I_{c^*}$ are all carried by 
a common probability space $(\Omega, \mathcal{F}, \mathbb{P})$.  
Let us denote by $\{\sF_t\}_t$ the filtration generated by $X(t,c^*)$ 
and $R(t,c^*)$. (We remark that if we assume that the $\nu_j$ are all distinct
then $\{\sF_t\}_t$ will be generated by $X(t,c^*)$ alone).  

The GT method involves defining for each $c \in I_{c^*}$  
a probability measure $\mathbb{P}_c$ on $(\Omega,\sF)$ which
satisfies the following condition.

{\bf Condition 1:} For each $c \in I_{c^*}$, (i) $\mathbb{P}_c$  
is absolutely continuous with respect to $\mathbb{P}$,
(ii) $\mathbb{P}_{c^*}=\mathbb{P}$ and (iii) for every 
bounded function $f:\posint^n \to \real$ 
\begin{equation}\label{eqn:change-of-measure}
\mathbb{E}f(X(T, c))  = \mathbb{E}_cf(X(T, c^*)),
\end{equation}
where $\mathbb{E}$ is the expectation with respect to 
$\mathbb{P}$ and $\mathbb{E}_c$ is the expectation with respect to $\mathbb{P}_{c}$.

Suppose such a family of probability measures $\mathbb{P}_c$ satisfying 
Condition 1 exist. 
Let us denote by $L(c,t)$ the Radon-Nykodim derivative
\begin{equation}
L(t,c) = \left.\frac{d \mathbb{P}_c}{d \mathbb{P}}\right|_{\sF_t}.
\end{equation}
Due to \eqref{eqn:change-of-measure}, the sensitivity can be written as
\[
\left.\frac{\partial}{\partial c}\right|_{c = c^*} \mathbb{E} f(X(T, c))
= \left.\frac{\partial}{\partial c}\right|_{c = c^*} \mathbb{E}_{c}
f(X(T,c^*)) = \left.\frac{\partial}{\partial c}\right|_{c = c^*}
\mathbb{E} [f(X(T, c^*)) L(T,c)].
\]

{\bf Condition 2:}
Suppose that the derivative 
\[
Z(t,c^*) = \left.\frac{\partial}{\partial c}\right|_{c = c^*} L(t,c)
\]
exists almost surely with respect to $\mathbb{P}_{c^*}$ and that 
the following commutation of derivative and expectation holds:
\begin{equation}\label{eq-cond2}
\mathbb{E} \left( \left.\frac{\partial}{\partial c}\right|_{c = c^*}
f(X(T,c^*)) L(T,c)\right) = \left.\frac{\partial}{\partial c}\right|_{c = c^*}
\mathbb{E} f(X(T,c^*)) L(T,c).
\end{equation}

This leads to the formula
\begin{equation}\label{eq-GT}
\left.\frac{\partial}{\partial c}\right|_{c = c^*} \mathbb{E} f(X(T, c))
 = \mathbb{E} [f(X(T, c^*))Z(T,c^*)].
\end{equation}
Thus if Conditions 1 and 2 are satisfied, the required sensitivity equals the expected value of the random 
variable $f(X(T, c^*))Z(T,c^*)$ and hence can be estimated via iid sample
estimation. Thus $\hat{s}_N$ given by  
\[\hat{s}_N = \frac{1}{N} \sum_{i = 1}^{N} f(X^{(i)}(T, c^*))
Z^{(i)}(T,c^*),\]
where $(X^{(i)}(T, c^*),Z^{(i)}(T,c^*))$ for $i=1,\dots,N$ are iid pairs of
samples, is the GT estimator for a sample size of $N$. We note that the 
simulation is carried out with respect to the probability measure
$\mathbb{P} = \mathbb{P}_{c^*} $. 
  
The GT estimator is unbiased but often has large variance unless $N$ is very
large \cite{wang2016efficiency}.
Nevertheless, due to its simplicity, GT has been widely applied for
sensitivity analysis in numerous areas such as chemical kinetics and operations research. 

While the GT method is widely used, we are not aware of theoretical studies 
on the validity of GT method.
In particular, in the area of stochastic reaction networks, no sufficient conditions have been provided to justify the method. 
Therefore, we believe a theoretical analysis of the method could provide a guideline about the applicability of GT to certain types of problems.    
In this paper, we aim to provide sufficient conditions 
that ensure Conditions 1 and 2 stated above.

\section{The validity of change of measure}\label{sec:change-of-measure}
\subsection{Change of intensity}
We explore some sufficient conditions that guarantee Condition 1 for the
change of measure. 
Our exposition here is based on the change of intensity theory in 
Section VI 2 of \cite{Bremaud}. We start with the probability space $(\Omega,\sF,\mathbb{P})$ 
on which the processes $X(t,c)$ and $R(t,c)$ are defined for 
$c \in I_{c^*} = (c^*-\epsilon,c^*+\epsilon)$. As before, we denote by $\sF_t$ the filtration
generated by $X(t,c^*)$ and $R(t,c^*)$. 
By definition, the counting process $R_j(t, c^*)$ has the $(\mathbb{P}, \mathcal{F}_t)$ predictable intensity $a_j(X(t-, c^*), c^*)$.
Now for any $c \in I_{c^*}$,  
we want to explicitly construct a probability measure $\mathbb{P}_c$ 
on $(\Omega,\sF)$ such that $R(t, c^*)$ admits the
$(\mathbb{P}_{c},\mathcal{F}_t)$ predictable intensity $a_j(X(t-, c^*),c)$. This is accomplished by defining the {\em likelihood ratio} process
$L(t,c)$ which under the right conditions will serve as the Radon-Nykodim
derivative.    
 
We first define an auxiliary ($m$ dimensional) process $\mu(t, c)$ 
under a certain assumption on the propensity functions as follows.
Given an arbitrary $c \in I_{c^*}$, we assume that for all $x \in \mathbb{Z}_+^n$, 
\begin{equation}
a_j(x, c^*) = 0 ~\text{if and only if}~ a_j(x, c) = 0, ~~~~j = 1, 2, \cdots, m.
\end{equation}
We note that this is assumption holds in particular when the propensity functions are of the
product form $a_j(x, c) = c_j b_j(x)$.
Now, based on this assumption, the following process is well-defined ($c^*$ is fixed). 
For each $c \in I_{c^*}$, we define
\begin{equation}\label{eqn:ratio}
\mu_j(t, c) = \frac{a_j(X(t-, c^*), c)}{a_j(X(t-, c^*), c^*)}, ~~~~j = 1,\cdots, m.
\end{equation}
In the case that $a_j(X(t-, c^*), c^*) = 0$, by assumption we have $a_j(X(t-,c^*), c) = 0$ as well, so we can simply define $\mu_j(t, c)$ to be any
strictly positive constant. 
We note that $\mu_j(t, c)$ is $\mathcal{F}_t$-predictable by its left
continuity, and moreover for each $t \geq 0$,  we have $ 0<\mu_j(t, c) < \infty$ almost surely. We shall make the extra assumption that 
$\mu_j(t,c)$ is bounded almost surely for each $c \in I_{c^*}$. In the case of 
the product form of propensity functions with $c=c_1$, we note that 
the boundedness assumption holds since $\mu_1(t,c) = c/c^*$ and $\mu_j(t,c)=1$
for $j=2,\dots,m$, that is the process $\mu(t,c)$ is a deterministic and constant in $t$. 

Next, following \cite{Bremaud}, we explicitly define the likelihood ratio process $L(t, c)$ as follows
\begin{equation}\label{eq-def-L}
L(t, c) = \prod_{j =1}^m\left( \left(\prod_{n=1}^{R_j(t, c^*)}\mu_j(T_j^n, c)\right) \exp{\left(\int_0^t (1 - \mu_j(s, c)) a_j(X(s, c^*), c^*) \,ds\right)}\right),
\end{equation}
where $T_j^n$ is the $n$-th jump time of $R_j(t,c^*)$.
By convention, we take the product $\prod_{n=1}^{R_j(t, c^*)}$ to be $1$ if
$R_j(t, c^*) = 0$. We remark that due to our non-explosivity assumption $R_j(t,c^*)$ is finite 
almost surely for each $t$ and thus $L(t,c)$ is well defined and satisfies 
$0<L(t,c)<\infty$ for each $t \geq 0$.. 
 
It can be shown that $L$ defined above is the solution of the equation \cite{Bremaud}
\begin{equation}\label{eq-L-SDE}
L(t, c) = 1+\sum_{j=1}^{m}\int_{(0, t]} L(s-, c)(\mu_j(s, c) - 1) d M_j(s,
  c^*),
\end{equation}
where 
\begin{equation}
M_j(t, c^*) = R_j(t, c^*) - \int_0^t a_j(X(s, c^*), c^*) ds.
\end{equation} 
From the non-explosivity assumption, we see that for all $t\geq 0$ and $j = 1,
\cdots, m$, 
\begin{equation}\label{eqn:integrability}
\int_0^t a_j(X(s, c^*), c^*) \,ds < \infty,~~\mathbb{P}~~~\text{a.s.},
\end{equation}
and hence $M(t, c^*)$ is an $m$-dimensional local martingale \cite{Bremaud}. 
We summarize some key results from \cite{Bremaud} as a lemmas. 
\begin{lemma}{\bf (Bremaud \cite{Bremaud}, Section VI, Theorem T2)}\label{lem:stochastic-expo}
Under the non-explosivity assumption (with respect to $\mathbb{P}$), for each $c \in I_{c^*}$, $L(t, c)$ is a $(\mathbb{P}, \mathcal{F}_t)$
nonnegative local martingale and hence a $(\mathbb{P}, \mathcal{F}_t)$ supermartingale.  
\end{lemma}

\begin{lemma}{\bf (Bremaud \cite{Bremaud}, Section VI, Theorem T3)} \label{lem:change-measure}
Suppose that $\mathbb{E}L(T, c) = 1$. 
Then $L(t,c)$ is a $(\mathbb{P}, \mathcal{F}_t)$ martingale over $[0,T]$.
Moreover, defining the probability measures $\mathbb{P}_c$ 
by the condition
$$\frac{d\mathbb{P}_c}{d\mathbb{P}} = L(t,c),$$ 
it follows that over $[0,T]$, $R_j(t, c^*)$ has the $(\mathbb{P}_c, \mathcal{F}_t)$-intensity
\[a_j(X(t-, c^*), c) = \mu_j(t, c) a_j(X(t-, c^*), c^*).\]
\end{lemma}

\begin{cor}\label{cor-change-measure}
Under the conditions of Lemma \ref{lem:change-measure},  for each bounded measurable function $f:\posint^n \to \real$ and each $c \in I_{c^*}$  
\begin{equation}\label{eqn:law}
\mathbb{E}f(X(T, c))  = \mathbb{E}_cf(X(T, c^*)).
\end{equation}
In other words, the $\mathbb{P}_c$ law of $X(t,c^*)$ is the same as the 
$\mathbb{P}$ law of $X(t,c)$.
\end{cor}
\begin{proof}
The law of a Markov process is uniquely determined by the Kolmogorov's
forward equation. The form of the forward equation depends on the
state dependence of the intensities of the jump event.  Since the $\mathbb{P}$
intensity of $R_j(t,c)$ is $a_j(X(t-,c),c)$ and the $\mathbb{P}_c$ intensity of
$R_j(t,c^*)$ is $a_j(X(t-,c^*),c)$, the state dependence are the same for both:
$a_j(x,c)$.   
Thus the Kolmogorov's forward equations 
for $p_1(x, t) = \mathbb{P}(X(t, c) = x)$ and 
$p_2(x, t) = \mathbb{P}_c(X(t,c^*) = x)$ are identical:
\[\frac{dp_i(x, t)}{dt} = \sum_{j=1}^m (p_i(x - \nu_j, t) a_j(x - \nu_j, c) -
p_i(x, t)a_j(x, c)), \quad x \in \mathbb{Z}_{+}^n,  \;\; i=1,2.\]
\end{proof} 

In summary, the condition $\mathbb{E}(L(T,c))=1$ in Lemma  \ref{lem:change-measure} 
ensures the validity of the change of measure. In the next section we 
provide some sufficient conditions for it to hold.

\subsection{Novikov type condition}\label{sec:Novikov}
In this section, we provide a Novikov type sufficient condition to ensure that
$L(t,c)$ is a martingale over $[0,T]$ or equivalently $\mathbb{E}(L(T,c))=1$. 
Under the product form of propensities $a_j(x,c)=c_j b_j(x)$, 
and without loss of generality, taking $c=c_1$, the likelihood ratio $L(t, c)$ can be written as 
\begin{equation}
L(t, c) = \left(\frac{c}{c^*}\right)^{R_1(t, c^*)} \exp\left(\int_0^t (c^* - c)b_1(X(s, c^*)) \, ds  \right). 
\end{equation}
We make the following useful observation. 
We have
\begin{equation}\label{eqn:L-estimate-1}
L(t, c) \leq \left(\frac{c}{c^*}\right)^{R_1(t, c^*)}
\end{equation}
for any $c \in I_{c^*}^+ = [c^*, c^* + \epsilon)$ and 
\begin{equation}\label{eqn:L-estimate-2}
L(t, c) \leq \exp\left(\int_0^t (c^* - c)b_1(X(s, c^*)) \, ds  \right)
\end{equation}
for any $c \in I_{c^*}^- = (c^* - \epsilon, c^*]$.
This simple observation turns out to be useful for our analysis.



\begin{theorem}\label{thm:GT-Novikov} 
Given $c \in I_{c^*}^+$, suppose that 
\begin{equation}\label{eqn:int-L-1}
\mathbb{E}\left[  \left(\frac{c}{c^*}\right)^{R_1(T, c^*)}  \right] < \infty,
\end{equation}
then $L(t, c)$ is a $(\mathbb{P}, \mathcal{F}_t)$ martingale over $[0,T]$.
\end{theorem}

\begin{proof}
By Lemma \ref{lem:stochastic-expo} $L(t,c)$ is a local martingale. 
Thus there exists a sequence $(\sigma_n)$ of increasing stopping times 
with $\sigma_n \uparrow \infty$ such that $L(t \wedge \sigma_n,c)$ is 
a $(\mathbb{P},\mathcal{F}_t)$ martingale for each $n$.  
Define the stopping times
\[\tau_n = \inf\{t \geq 0~|~ R_1(t, c^*) \geq n\}.\]
By the non-explosivity assumption, $\tau_n \uparrow \infty$.
We define the stopped processes 
\[L_n(t, c) = L(t \wedge \sigma_n \wedge \tau_n, c).\]
Now for each $n$,  as a stopped martingale, $L_n(t, c)$ is a $\mathcal{F}_t$
martingale and hence $\mathbb{E}L_n(T, c) = 1$. 
By the estimates in \eqref{eqn:L-estimate-1},
\[L_n(T, c) \leq \left(\frac{c}{c^*}\right)^{R_1(T \wedge \sigma_n \wedge \tau_n, c^*)} \leq \left(\frac{c}{c^*}\right)^{R_1(T, c^*)}.\] 
Hence, the integrability condition \eqref{eqn:int-L-1} implies that $\mathbb{E}L(T,c) = 1$ by the dominated 
convergence theorem and therefore $L(t,c)$ is a martingale over $[0,T]$.
\end{proof}

Similar result can be reproduced for the case $c \in I_{c^*}^-$ using the estimates \eqref{eqn:L-estimate-2}. 
\begin{theorem}\label{thm:GT-Novikov-2}
Given $c \in I_{c^*}^-$, suppose that
\begin{equation}\label{eqn:int-L-2}
\mathbb{E}\left[\exp\left((c^*-c) \int_0^T b_1(X(s, c^*)) \, ds\right)\right] < \infty,
\end{equation}
then $L(t, c)$ is a $(\mathbb{P}, \mathcal{F}_t)$ martingale over $[0,T]$.
\end{theorem}
\begin{proof}
Define the stopping times $\tau_n$ by
\[
\tau_n = \inf\left\{t \geq 0 \, : \, \int_0^t b_1(X(s,c^*)) ds \geq n\right\}.
\] 
The rest of the proof is similar to that of Theorem \ref{thm:GT-Novikov}. 
\end{proof}

One can get rid of the time integral by verifying the following stronger condition. 
\begin{cor}\label{thm:GT-Novikov-Cor1}
If there exists $\epsilon > 0$ such that
\begin{equation}\label{eqn:GT-Novikov-Cor1}
\sup_{s \leq T} \mathbb{E} \left[  e^{\epsilon T b_1(X(s, c^*))}   \right] < \infty,
\end{equation} 
then 
$L(t, c)$ is a $(\mathbb{P}, \mathcal{F}_t)$ martingale over $[0,T]$ 
for any $c \in I_{c^*}^-=(c^*-\epsilon,c^*)$.
\end{cor}
\begin{proof}
Rearranging the right hand side of \eqref{eqn:int-L-2} and applying Jensen's inequality to the time average integral inside the bracket, we obtain
\begin{equation}
\begin{split}
\mathbb{E}\left[ \exp\left(\frac{1}{T}\int_0^T (c^* - c) T b_1(X(s, c^*))\,ds\right)\right]
\leq \mathbb{E}\left[  \frac{1}{T}\int_0^T   e^{(c^* - c) T b_1(X(s, c^*))}     \, ds     \right]
\end{split}
\end{equation}
By Fubini's Theorem we can commute the time integral and the expectation and hence it suffices to show that 
\[\sup_{s\leq T}\mathbb{E}\left[e^{\left( \epsilon T b_1(X(s, c^*))\right)}\right]\]
is finite for $\epsilon$ satisfying $\epsilon > c^* - c$.
\end{proof}


\section{Differentiation inside the integral}\label{sec:swap}
In this section, we provide a sufficient condition for the commutation 
\eqref{eq-cond2} of Condition 2 via the use of
Theorem \ref{thm:Glynn} in appendix. Referring to Theorem \ref{thm:Glynn}, 
we take $G$ to be 
\[G(c) = f(X(T, c^*)) L(T, c).\]
We shall assume the product form with $c=c_1$. Then it follows at once 
from \eqref{eq-def-L} that
\[\frac{\partial}{\partial c}\ln L(T, c) = \frac{1}{c}R_1(T, c^*) - \int_0^T  b_1(X(s, c^*))  \,ds,\]
hence
\[\frac{\partial}{\partial c} L(T, c) = L(T, c)  \left(\frac{1}{c}R_1(T, c^*) - \int_0^T  b_1(X(s, c^*))  \,ds \right)= \frac{1}{c}L(T, c) M_1(T, c^*).\]
Then a Lipschitz constant $K(\omega)$ (independent of $c$) for $G$ on the interval $I_{c^*}$ is 
\[K= |f(X(T, c^*))M_1(T, c^*)| \sup_{c \in I_{c^*}} \frac{1}{c}L(T, c).\]

We first consider $c \in I_{c^*}^+$ (the right hand sensitivity) , in which case we have 
\[L(T, c) \leq \left(\frac{c}{c^*}\right)^{R_1(T, c^*)} \leq \left(\frac{c^* + \epsilon}{c^*}\right)^{R_1(T, c^*)} . \]
Hence in order to justify the integrability of $K$, it suffices to show that 
\[f(X(T, c^*))M_1(T, c^*)\left(\frac{c^* + \epsilon}{c^*}\right)^{R_1(T, c^*)} \] 
is integrable. We also note that we may shrink the interval $I_{c^*}=(c^*-\epsilon,c^*+\epsilon)$ 
to be as small as we wish.

\begin{theorem}\label{thm:swap-right}
Assuming the product form \eqref{eq-mass-action} with $c=c_1$, suppose the following conditions are satisfied:
\begin{itemize}
\item $\mathbb{E}  [|f(X(T, c^*))|^3] < \infty$;
\item there exists $\epsilon > 0$ such that,
\begin{equation}\label{eqn:main-integrability-right}
\mathbb{E}\left[  \left(\frac{c^* +\epsilon}{c^*}\right)^{R_1(T, c^*)}  \right] < \infty.
\end{equation}
\end{itemize}
Then 
\begin{equation*}
\lim_{c \to c^{*+}} \frac{\mathbb{E} f(X(T, c)) - \mathbb{E} f(X(T, c^*))}{c - c^*} = 
\mathbb{E} \left[f(X(T, c^*)) \lim_{c\to c^{*+}}\frac{L(T, c) - L(T, c^*)}{c-c^*}  \right].
\end{equation*}
\end{theorem}

\begin{proof}
First we note that by Theorem \ref{thm:GT-Novikov},
\eqref{eqn:main-integrability-right} implies the validity of the change of
measure for $c \in I_{c^*}^+$. 

Now we need to verify the limit using Theorem \ref{thm:Glynn}.  
Using the inequality $3abc \leq a^3 + b^3 + c^3$, we can separate the terms and provide
the following sufficient conditions, 
\begin{equation}\label{eqn:3-conditions}
\begin{split}
&\mathbb{E} [ |f(X(T, c^*))|^3] < \infty,\\
&\mathbb{E} [ |M_1(T, c^*)|^3] < \infty,\\
&\mathbb{E} \left[ \left(\frac{c^* + \epsilon}{c^*}\right)^{3R_1(T, c^*)} \right] < \infty.
\end{split}
\end{equation}
It is sufficient to show that the third condition implies the second condition. 
Since the quadratic variation of the local martingale $M_1(t,c^*)$ is $R_1(t,c^*)$, by the Burkholder-Davis-Gundy (BDG) inequality \cite{protter2005stochastic}, 
\[\mathbb{E} ( |M_1(T, c^*)|^4) \leq C \mathbb{E} [R_1(T,c^*)^2]. \]
for some constant $C$.
It is obvious that the right hand side is integrable given the second
condition. Since $|M_1(T, c^*)|^4$ is integrable the result follows. 
\end{proof}

Similarly, for the left hand side sensitivity, we have 
\[L(T, c) \leq \exp\left(\int_0^T (c^* - c)b_1(X(s, c^*)) \, ds  \right) \leq  \exp\left(\int_0^T \epsilon b_1(X(s, c^*)) \, ds  \right)\]
for $c \in I_{c^*}^-$.
Hence, the Lipschitz constant is proportional to 
\[|f(X(T, c^*))M_1(T, c^*)| \exp\left(\int_0^T \epsilon b_1(X(s, c^*)) \, ds  \right).\]
It boils down to verifying the following three integrability conditions,
\begin{equation}\label{eqn:3-conditions-left}
\begin{split}
&\mathbb{E} [ |f(X(T, c^*))|^3] < \infty,\\
&\mathbb{E} [ |M_1(T, c^*)|^3] < \infty,\\
&\mathbb{E}\left[\exp\left(\int_0^T 3\epsilon b_1(X(s, c^*)) \, ds  \right)\right] < \infty.
\end{split}
\end{equation}
We have the following result concerning the left hand side sensitivity.

\begin{theorem}\label{thm:swap-left}
Assuming the product form \eqref{eq-mass-action} with $c=c_1$, suppose
further that
\begin{itemize}
\item $\mathbb{E} [ |f(X(T, c^*))|^3] < \infty$;
\item there exists $\epsilon > 0$ such that,
\begin{equation}\label{eqn:main-integrability-left}
\mathbb{E}\left[\exp\left(\int_0^T \epsilon b_1(X(s, c^*)) \, ds  \right)\right] < \infty.
\end{equation}
\end{itemize}
Then
\begin{equation*}
\lim_{c \to c^{*-}} \frac{\mathbb{E} f(X(T, c)) - \mathbb{E} f(X(T, c^*))}{c - c^*} = 
\mathbb{E} \left[f(X(T, c^*)) \lim_{c\to c^{*-}}\frac{L(T, c) - L(T, c^*)}{c-c^*}  \right].
\end{equation*}
\end{theorem}

\begin{proof}
We only need to show the second integrability condition in \eqref{eqn:3-conditions-left} holds. 
Note that $M_1(t, c^*)$ in the second term is a local martingale, we apply the  (BDG) inequality such that
\[\mathbb{E} \{ |M_1(T, c^*)|^4\} \leq C \mathbb{E} [R_1(T, c^*)^2] \]
for some constant $C$. 
Hence, it is sufficient to verify that $\mathbb{E} [R_1(T, c^*)^2] < \infty$.
Applying the BDG inequality again to $\mathbb{E}[M_1(T, c^*)^2]$, there exists some constant 
$\bar{C}$ such that
\[\mathbb{E}[M_1(T, c^*)^2] \leq \bar{C} \mathbb{E} [R_1(T, c^*)] = \bar{C} \mathbb{E} \left[\int_0^T a_1(X(s, c^*), c^*) \, ds\right] < \infty.\] 
Owing to the simple inequality $(a+b)^2 \leq 2(a^2 + b^2)$,
\[\mathbb{E} [R_1(T, c^*)^2] \leq 2 \mathbb{E} \left[\left(\int_0^T a_1(X(s, c^*), c^*)\, ds\right)^2\right] + 2\mathbb{E}[M_1(T, c^*)^2] < \infty.\]
\end{proof}

\begin{remark}\label{rem-rightleft}
We note that the conditions of Theorem \ref{thm:swap-right} guarantee 
the existence of the right hand derivative (sensitivity) and that the 
GT method would provide an unbiased estimator of it. 
Likewise for Theorem \ref{thm:swap-left}. However, the conditions of these
theorems include the restrictive exponential integrability conditions
\eqref{eqn:main-integrability-right} and  \eqref{eqn:main-integrability-left}. 
We do know from \cite{APA} that the existence of sensitivity can be guaranteed
under milder conditions. Thus, if we assume the existence of sensitivity 
at $c^*$, then verification of either the conditions of Theorem \ref{thm:swap-right} 
or those of Theorem \ref{thm:swap-left} can guarantee 
the validity of the GT method. This will be our focus in the next section. 
\end{remark}

\section{Sufficient conditions in terms of the network}
The conditions of of Theorems \ref{thm:swap-right} and \ref{thm:swap-left} 
are not directly stated in terms of a chemical reaction network. 
A chemical reaction network is characterized by the stoichiometric matrix
$\nu$ (whose columns are the vectors $\nu_j$) and the propensity functions 
$a_j(x,c)$ which we assume to be of the product form
\eqref{eq-mass-action}.
 
We shall focus on the case that $f:\posint^n \to \real$ is of polynomial
growth. This may be stated by the condition that there exists $C>0$ and $r \in
\posint$ such that 
\[
|f(x)| \leq C (1 + \|x\|^r) \quad \forall x \in \posint^n.
\]
In this case, there exist multiple results in the literature that guarantee 
the condition that $|f(X(T,c^*)|^3$ is integrable \cite{rathinam2015Quarterly,  gupta2014scalable, engblom2014AppliedMath}. 

On the other hand, the
exponential integrability conditions \eqref{eqn:main-integrability-right} or
\eqref{eqn:main-integrability-left} are harder to satisfy. 
When $b_1(x)$ is linear, the condition \eqref{eqn:GT-Novikov-Cor1} is implied
by the conditions for the 
{\em uniform light-tailedness} property presented in \cite{gupta2014scalable},
and since \eqref{eqn:GT-Novikov-Cor1} implies
\eqref{eqn:main-integrability-left}, this provides a sufficient condition 
for the validity of the GT method. However, the {\em uniform light-tailedness}
property presented in \cite{gupta2014scalable} may be too stringent as it
is concerned with the supremum over the infinite time horizon $[0,\infty)$.

\begin{remark}\label{rem-mono-system}
  When the reaction network consists only of reactions of the simple
  monomolecular form $S_i \to S_j$ or $\emptyset \to S_i$ or $S_i \to
  \emptyset$ with the stochastic mass action form of propensities
  \cite{Gillespie77}, then the species populations can be shown to be
  sums of multinomial and Poisson random variables \cite{jahnke2007solving}.
  Since Poisson and multinomial random variables $X$ satisfy exponential
  integrability ($\mathbb{E}(e^{\epsilon |X|})<\infty$ for each $\epsilon>0$),
  and moreover, since the propensities are linear, \eqref{eqn:GT-Novikov-Cor1}
  and hence \eqref{eqn:main-integrability-left} may be verified.
\end{remark}

Given a fixed initial state $x_0 \in \mathbb{Z}_+^n$ let $\mathcal{S}_{x_0}
\subset \mathbb{Z}_+^n$ denote the set of all states that can be reached by the 
process starting at $x_0$. Thus $\mathcal{S}_{x_0}$ is the effective state space of 
the process $X(t,c)$ and it may be finite or infinite.
We remark that if $\mathcal{S}_{x_0}$ is finite, then the validity 
of the GT method follows trivially. As a consequence, the
case of interest to us is when $\mathcal{S}_{x_0}$ is infinite. 

In the rest of this section, we present some sufficient conditions that imply
\eqref{eqn:main-integrability-right} (in the context of reaction networks).
Our main idea is to focus on some reactions $R_j$ which may be
easily shown to satisfy the
condition $\mathbb{E}(e^{\epsilon R_j(t)})<\infty$ for some $\epsilon>0$,
and then to bound other reactions $R_k$ in terms of $R_j$, that is, to obtain
an almost sure bound of the form
\[
R_k(t) \leq a + b R_j(t).
\]
The non-negativity of the species population $X(t)$ in a reaction networks
implies linear inequality relations among the reaction counts $R_j(t)$
via
\[
X(t) = x_0 + \nu R(t) \geq 0.
\]
Thus we expect to bound some reactions in terms of the other. 
With this in mind, we define a reaction $j$ to be {\it unconsuming} if its associated
stoichiometric vector $\nu_j$ is nonnegative: $\nu_{ij} \geq 0$ for
$i=1,\dots,n$. We shall call a reaction {\it consuming} if it is not
unconsuming. We shall use $\mathcal{C}$ to denote the indices $j$ of all 
consuming reactions and use $\mathcal{U}$ to denote the indices $j$ of all
unconsuming reactions.

We note that an unconsuming reaction may not be bounded above in terms of
another reaction. Motivated by this fact, we introduce the following useful property regarding the unconsuming reactions
which may or may not hold in a given network.
\begin{assume}\label{assume-unconsuming-reaction}
Given $t > 0$, for any $j \in \mathcal{U}$, there exists $\epsilon>0$ such that 
\[
\mathbb{E}(e^{\epsilon R_j(t, c)}) < \infty.
\]
\end{assume}
\begin{remark}\label{remark:bounded-unconsuming-reaction}
Property \ref{assume-unconsuming-reaction} is readily satisfied if all the propensities of the unconsuming reactions are bounded on $\mathcal{S}_{x_0}$. 
To see this, suppose that there exists $K>0$ such that the propensity $a_j(x,c)$  satisfies 
\[
a_j(x,c) \leq K \quad \forall x \in \mathcal{S}_{x_0}.
\]
Then $\mathbb{E}(e^{\epsilon  R_j(t,c)}) < \infty$ for every $\epsilon>0$.
By the random time change representation, 
\[
R_j(t,c)=Y_j \left(\int_0^t a_j(X(s,c),c) ds\right) \leq Y_j(K t).
\]
Thus $R_j(t,c) \leq Y_j(K t)$ and consequently
\[
\mathbb{E}(e^{\epsilon R_j(t,c)}) \leq 
\mathbb{E}(e^{\epsilon Y_j(K t)}).
\] 
The result follows from the fact that $\mathbb{E}(e^{\epsilon Y_j(Kt)}) < \infty$
for every $\epsilon>0$.
\end{remark}
We note that, if all the unconsuming reactions are of the form
$\emptyset \to S_i$ or $S_j \to S_j + S_i$ 
where the copy number of species $S_j$ is bounded on $\mathcal{S}_{x_0}$, then
Remark~\ref{remark:bounded-unconsuming-reaction} readily applies.

Without loss of generality, we are interested in the parameter $c=c_1$ 
at a nominal value $c^*>0$, i.e., the parameter of reaction $1$.
If reaction $1 \in \mathcal{U}$, i.e., reaction $1$ is unconsuming, then 
the required bound \eqref{eqn:main-integrability-right} is automatically
satisfied if Property~\ref{assume-unconsuming-reaction} holds.
If reaction $1 \in \mathcal{C}$ and its propensity is unbounded on $\mathcal{S}_{x_0}$,
it may yet be possible to bound $R_1(t)$ above in terms of the unconsuming
reactions and furthermore if Property~\ref{assume-unconsuming-reaction} holds, then
\eqref{eqn:main-integrability-right} holds. 
\begin{example}
As a motivating example, let us consider the chemical kinetics example with two species and three reactions:
\[ S_1 \to S_2, \quad S_1 + S_2 \to \emptyset,  \quad  \emptyset \to S_1 \]
with 
\[
\nu_1 = (-1,1)^{\textrm{T}}, \;\; \nu_2=(-1,-1)^{\textrm{T}},  \;\; \nu_3=(1,0)^{\textrm{T}}.
\]
The propensities are given by $a_1(x)= c_1 x_1$, $a_2(x)=c_2 x_1 x_2$ and
$a_3(x)=c_3$. We note that $\mathcal{S}_{x_0}$ is unbounded. Since the propensity of
reaction channel $3$ is constant, by
Remark~\ref{remark:bounded-unconsuming-reaction}, there exists $\epsilon>0$
such that
\[
\mathbb{E}(e^{\epsilon R_3(t)}) < \infty.
\]
Thus \eqref{eqn:main-integrability-right} holds when the parameter of interest is
$c_3$. However, the other two reaction channels have unbounded propensities.

In order to bound $R_1(t)$ and $R_2(t)$ in terms of $R_3(t)$, 
we make explicit use of the fact that species population process $X(t)$ remains 
nonnegative. The process at any time $t \geq 0$ satisfies \eqref{eq-X-R}
\[
X(t) = x_0 + \nu R(t)
\] 
and since $X(t) \geq 0$ we have that $x_0 + \nu R(t) \geq 0$. 
We readily see that 
\begin{equation}\label{eqn:bound-consuming-reaction}
R_1(t) \leq R_1(t) + R_2(t) \leq x_{0, 1} + R_3(t).
\end{equation}
Now \eqref{eqn:bound-consuming-reaction} readily implies that 
\[
\mathbb{E}(e^{\epsilon R_i(t)}) < \infty
\]
for $i=1,2$, showing the condition \eqref{eqn:main-integrability-right}
with respect to all three parameters. 
\end{example}

This example suggests the possibility that the relation $x_0 + \nu R(t) \geq
0$ may imply that $R_1(t)$ (the reaction channel of interest) is bounded above 
in terms of a positive affine combination of those unconsuming reactions which satisfy Property~\ref{assume-unconsuming-reaction}. In general, this determination could be
made as follows. 

Suppose that $1 \in \mathcal{C}$.
The fact that $X(t) \geq 0$ can be expressed by
\[
x_0 +  \sum_{j \in \mathcal{C}} \nu_j R_j(t) + \sum_{j \in \mathcal{U}}
\nu_j R_j(t) \geq 0.
\]
For $j \in \mathcal{U}$ let $\mu_j \in \real^n$ be defined by 
$(\mu_j)_i = \max\{0,\nu_{i,j}\}$. Then $X(t) \geq 0$ implies 
\[
x_0 + \sum_{j \in \mathcal{C}} \nu_j R_j(t) + \sum_{j \in \mathcal{U}}
\mu_j R_j(t) \geq 0.
\]   
Letting 
\[
y=x_0 + \sum_{j \in \mathcal{U}} \mu_j R_j(t)
\]
and noting that $y \geq x_0$,
motivates the linear programming (feasibility) problem:
\begin{equation}\label{eq-lin-program}
y + \sum_{j \in \mathcal{C}} \nu_j \xi_j \geq 0, \;\; \xi \geq 0,
\end{equation}
where $\xi \in \real^n$. The feasible region for $\xi$ is given by 
a convex polytope $\mathcal{R}_y$ which may be unbounded. If $\mathcal{R}_y$ 
is bounded in the $\xi_1$ direction, then one can obtain an upper bound 
for $R_1(t)$ as an affine combination of $R_j(t)$ for $j \in \mathcal{U}$. 
We note that whether $\mathcal{R}_y$ is bounded in the $\xi_1$ direction 
or not depends only on $\nu_j$ for $j \in \mathcal{C}$ and not on $y$. 
   
Then Property~\ref{assume-unconsuming-reaction} can be used to obtain
\eqref{eqn:main-integrability-right}. We summarize this discussion as a theorem.
\begin{theorem}\label{thm-main}
Given a non-explosive chemical reaction network with product form propensity functions,
suppose that Property~\ref{assume-unconsuming-reaction} holds and 
that either reaction $1$ is unconsuming
or the feasible region of the linear program \eqref{eq-lin-program} is 
bounded in the first coordinate $\xi_1$.
Then \eqref{eqn:main-integrability-right} holds. 
\end{theorem}

We illustrate the application of Theorem~\ref{thm-main} by examining the following  network that models gene expression.

%
%
%

\begin{example}
Let us consider the following system of gene expression
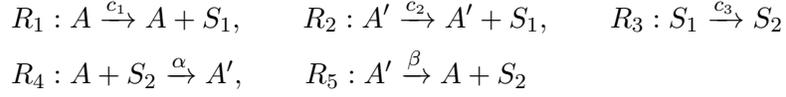
\begin{equation*}
\begin{split}
&R_1 : A \xrightarrow{c_1} A + S_1, \qquad R_2 : A^{\prime} \xrightarrow{c_2} A^{\prime} + S_1, \qquad R_3 : S_1 \xrightarrow{c_3} S_2 \\
&R_4 : A + S_2 \xrightarrow{\alpha} A^{\prime}, \qquad R_5 : A^{\prime} \xrightarrow{\beta} A + S_2
\end{split}
\end{equation*}
with stoichiometric vectors
\begin{equation*}
\begin{split}
&\nu_1 = (0, 0, 1, 0)^{\textrm{T}}, \quad \nu_2 = (0, 0, 1, 0)^{\textrm{T}}, \quad \nu_3 = (0,0,-1,1)^{\textrm{T}}, \\
&\nu_4 = (-1, 1, 0, -1)^{\textrm{T}}, \quad \nu_5 =(1, -1, 0, 1)^{\textrm{T}}.
\end{split}
\end{equation*}
Denote the population of $(A, A^{\prime},  S_1, S_2)$ by
$(x_1, \ldots, x_4)$ and hence the propensities are
\begin{equation*}
\begin{split}
a_1(x, c) = c_1 x_1, ~a_2(x, c) = c_2 x_2,  ~a_3(x, c) = c_3 x_3,
~a_4(x, c) = \alpha x_1 x_4, ~a_5(x, c) = \beta x_2.
\end{split}
\end{equation*}
Note that the total population of $A$ and $A^{\prime}$ is preserved and hence $a_1, a_2$
and $a_5$ are bounded on $\mathcal{S}_{x_0}$. Therefore, condition \eqref{eqn:main-integrability-right}
with respect to $c_1, c_2$ and $\beta$ by Remark~\ref{remark:bounded-unconsuming-reaction}.
The propensities associated with
$R_3$ and $R_4$ are not bounded as the populations of $S_1$ and $S_2$
are unbounded. However, note that

\[X_1(t) = x_{0, 1} - R_4(t) + R_5(t) \geq 0\]
and therefore
\[
\mathbb{E}(e^{\epsilon R_4(t)}) \leq \mathbb{E}(e^{\epsilon(x_{0,1} + R_5(t))}) < \infty
\]
for some $\epsilon > 0$
since reaction~$5$ has bounded propensity. 
Similarly, note that
\[
X_3(t) = x_{0, 3} + R_1(t) + R_2(t) - R_3(t) \geq 0
\]
and therefore
\[
\mathbb{E}(e^{\epsilon R_3(t)}) \leq \mathbb{E}(e^{\epsilon(x_{0, 3} + R_1(t) + R_2(t))}) < \infty.
\]
Hence, condition \eqref{eqn:main-integrability-right}
holds with respect to $c_3$ and $\alpha$ as well. 
\end{example}

We note that the usefulness of Theorem~\ref{thm-main} for
reactions that have unbounded propensities relies on the
condition involving the linear programming problem \eqref{eq-lin-program}.
We illustrate an example where this condition fails. That is, it is not
possible to bound a reaction with unbounded propensity in terms of
a reaction with bounded propensity. 

\begin{example}\label{eg:lin-prog-fail}
  Consider the reaction system
  \[
  R_1: \, \emptyset \to S_1 \quad R_2:\, 2 S_1 \to 2 S_2 \quad R_3:\, 2 S_2 \to 2
  S_1,
  \]
with stochastic mass action propensities $a_1(x) = c_1$, $a_2(x)=c_2
x_1(x_1-1)/2$ and $a_3(x)=c_3 x_2 (x_2-1)/2$. The stoichiometric vectors are
\[
\nu_1=(1,0)^{\textrm{T}}, \quad \nu_2=(-2,2)^{\textrm{T}}, \quad \nu_3=(2,-2)^{\textrm{T}}.
\]
The propensity of the first
reaction is bounded and hence the exponential integrability
$\mathbb{E}(e^{\epsilon R_1(t)})<\infty$ holds for each $\epsilon>0$.
However, if we are interested in the sensitivities with respect to $c_2$ or
$c_3$, then we wish to verify $\mathbb{E}(e^{\epsilon R_j(t)})<\infty$ for
$j=2$ or $3$. Since $\mathcal{S}_{x_0}$ is unbounded, the propensities $a_2$ and $a_3$
are unbounded. Thus, it will be desirable to bound $R_2(t)$ and/or $R_3(t)$
in terms of $R_1(t)$.

The non-negativity of states, $X_1(t) \geq 0$ and $X_2(t) \geq 0$, again
yields
\[
\begin{aligned}
  X_1(t) &= x_{0, 1} + R_1(t) - 2 R_2(t) + 2 R_3(t) \geq 0,\\
  X_2(t) &= x_{0, 2} + 2 R_2(t) - 2 R_3(t) \geq 0.
\end{aligned}
\]
These inequalities permit arbitrarily large values for $R_2(t)$ and $R_3(t)$
for a given value of $R_1(t)$. Thus, it is not possible to bound $R_2(t)$
or $R_3(t)$ in terms of $R_1(t)$. This will be reflected in the failure of
the condition
on the linear programming problem (after a reordering of reaction indices)
stated in Theorem~\ref{thm-main}. 
\end{example}

Next, we focus on the situation where the system has an unconsuming reaction
with unbounded propensity.
Even if we may bound all other reactions in terms
of this unconsuming reaction, one still needs to verify the 
Property~\ref{assume-unconsuming-reaction} for this unconsuming reaction. 
We describe a coupling strategy inspired by
\cite{kurtz1982representation, CFD} and is based on the RTC~\eqref{eqn:random-time-change} that allows us to verify 
Property~\ref{assume-unconsuming-reaction} for a class of networks where
unconsuming reactions may have unbounded propensities. 
We motivate the strategy through two examples. 

\begin{example}\label{eg:coupling-1}
We consider the reaction system 
\[ R_1:\, S \xrightarrow{c_1} 2S, \qquad R_2:\,2S \xrightarrow{c_2} \emptyset\]
with stoichiometric vectors $\nu_1 = 1, \nu_2 = -2$ and initial population
$x_0$. We take $a_1(x)= c_1 x_1$ and $a_2(x)=c_2 x_2 (x_2-1)/2$. 

We note that the propensities are both unbounded. Our first aim is to justify that $\mathbb{E}(e^{\epsilon R_1(t, c)}) < \infty$ for
some $\epsilon>0$. 
Let us denote $X(t)$ the population of species $S$ in the above reaction system.
In the meanwhile, we consider another reaction system with a single reaction channel
\[ \widetilde{R}_1:\, \widetilde{S} \xrightarrow{c_1} 2\widetilde{S}\]
with the same initial population $x_0$
and denote the population of $\widetilde{S}$ by $\widetilde{X}(t)$. 
Note that by the random time change representation \eqref{eqn:random-time-change} we can couple the processes $X(t)$ and $\widetilde{X}(t)$ in the same probability space using 
the following coupling
\begin{equation}\label{eqn:coupling-1}
\begin{split}
X(t) = x_0 + &Y_1\left(\int_0^t a_1(X(s)) \wedge a_1(\widetilde{X}(s))\, ds\right)\\
                 +& Y_2\left(\int_0^t a_1(X(s)) - a_1(X(s)) \wedge a_1(\widetilde{X}(s)) \, ds\right)\\
                 - & 2 Y_3\left( \int_0^t a_2(X(s)) \, ds \right) 
\end{split}
\end{equation}
and
\begin{equation}\label{eqn:coupling-2}
\begin{split}
\widetilde{X}(t) = x_0 + &Y_1\left(\int_0^t a_1(X(s)) \wedge a_1(\widetilde{X}(s))\, ds\right)\\
                 +& Y_2\left(\int_0^t a_1(\widetilde{X}(s)) - a_1(X(s)) \wedge a_1(\widetilde{X}(s)) \, ds\right),
\end{split}
\end{equation}
where $Y_i, i = 1, 2, 3$ are independent unit rate Poisson process carried by
$(\Omega, \mathcal{F}, \mathbb{P})$.
We note that
\[
\begin{aligned}
R_1(t) = &Y_1\left(\int_0^t a_1(X(s)) \wedge a_1(\widetilde{X}(s))\, ds\right)\\
                 +& Y_2\left(\int_0^t a_1(X(s)) - a_1(X(s)) \wedge a_1(\widetilde{X}(s)) \, ds\right)\\
\end{aligned}
\]
and
\[
\begin{aligned}
\widetilde{R}_1(t) = &Y_1\left(\int_0^t a_1(X(s)) \wedge a_1(\widetilde{X}(s))\, ds\right)\\
                 +& Y_2\left(\int_0^t a_1(\widetilde{X}(s)) - a_1(X(s)) \wedge
                 a_1(\widetilde{X}(s)) \, ds\right). 
\end{aligned}
\]

We claim that for any $t > 0$, $X(t) \leq \widetilde{X}(t)$ almost surely.
Define $T_n$ the $n$-th jump time of the processes $X(t)$ and
$\widetilde{X}(t)$ combined. Set $T_0=0$. We note that $X(T_0)=\widetilde{X}(T_0)=x_0$.
Suppose $X(T_j) \leq \widetilde{X}(T_j)$ for $j=0,\dots,n-1$ for some $n$. 
Since $X(t)$ and $\widetilde{X}(t)$ are constant between two successive jumps, 
$X(t) \leq \widetilde{X}(t)$ for all $t < T_{n}$ and hence by monotonicity
of $a_1$, we have $a_1(X(t)) \leq a_1(\widetilde{X}(t))$ for all $t < T_{n}$.
Thus
\[
\int_0^{T_n} (a_1(X(s)) - a_1(\widetilde{X}(s)) \wedge a_1(X(s)) \, ds =0.
\]
Hence
\[
X(T_{n}) = x_0 + Y_1\left(\int_0^{T_{n}} a_1(X(s)) \, ds\right) + Y_1(0) -  2 Y_3\left( \int_0^{T_n} a_2(X(s)) \, ds \right) 
\]
and 
\[
\widetilde{X}(T_n) = x_0 +  Y_1\left(\int_0^{T_n} a_1(X(s)) \, ds\right) + Y_2\left(\int_0^{T_n} a_1(\widetilde{X}(s)) - a_1(X(s))  \, ds\right).
\] 
On account of the fact that a unit Poisson process starts at zero,
$Y_1(0)=0$, and it is clear that $X(T_n) \leq \widetilde{X}(T_n)$, completing
the induction. 

Now since $a_1(X(t)) \leq a_1(\widetilde{X}(t))$ for any $t>0$, we deduce
from~\eqref{eqn:coupling-1} and~\eqref{eqn:coupling-2} that 
\[
Y_1\left(\int_0^t a_1(X(s)))\, ds\right) 
=   R_1(t)
\]
and
\[
Y_1\left(\int_0^t a_1(X(s)))\, ds\right) + 
Y_2\left(\int_0^t a_1(\widetilde{X}(s)) - a_1(X(s))  \, ds\right)
=   \widetilde{R}_1(t),
\]
showing $R_1(t) \leq \widetilde{R}_1(t)$. This in turn implies that $\mathbb{E}(e^{\epsilon R_1(t)}) \leq \mathbb{E}(e^{\epsilon \widetilde{R}_1(t)})$.
Consequently, the desired exponential integrability of $R_1(t)$ follows from the fact that the reaction count $\widetilde{R}_1(t)$ of auxiliary system $S \to 2S$ is a negative binomial process and hence there exists $\epsilon > 0$ such that $\mathbb{E}(e^{\epsilon \widetilde{R}_1(t)}) < \infty$ (see Appendix~\ref{app:negative-binomial}). 
Finally, the exponential integrability \eqref{eqn:main-integrability-right} holds  with respect to $c_2$ since 
\[
X(t) = x_0 + R_1(t) - 2R_2(t) \geq 0
\]
and hence $R_2(t)$ can be bounded in terms of $R_1(t)$ whose exponential integrability 
has already been established.
\end{example}

The next example extends the coupling argument further. 

\begin{example}\label{eg:coupling-2}
Let us consider the two species Lotka-Volterra model \cite{gupta2014scalable}, namely,
\begin{equation*}
\begin{split}
R_1: \, \emptyset \xrightarrow{\alpha_1} S_1, \quad R_2: \, S_1 \xrightarrow{\beta_1} 2 S_1,
\quad R_3:\,S_1 + S_2 \xrightarrow{\gamma_{12}} S_2, \quad R_4:\, S_1 \xrightarrow{\delta_1} \emptyset, \\
R_5: \, \emptyset \xrightarrow{\alpha_2} S_2, \quad R_6: \, S_2 \xrightarrow{\beta_2} 2 S_2,
\quad R_7:\,S_2 + S_1 \xrightarrow{\gamma_{21}} S_1, \quad R_8:\, S_2 \xrightarrow{\delta_2} \emptyset.
\end{split}
\end{equation*}
whose stoichiometric vectors are
\begin{equation*}
\begin{split}
\nu_1 = (1, 0)^{\textrm{T}}, \quad \nu_2 = (1, 0)^{\textrm{T}}, \quad \nu_3 = (-1, 0)^{\textrm{T}}, \quad \nu_4 = (-1, 0)^{\textrm{T}}\\
\nu_5 = (0, 1)^{\textrm{T}}, \quad \nu_6 = (0, 1)^{\textrm{T}}, \quad \nu_7 = (0, -1 )^{\textrm{T}}, \quad \nu_8 = (0, -1)^{\textrm{T}}
\end{split}
\end{equation*}
Since the propensity of reaction channel $1$ is constant,
$\mathbb{E}(e^{\epsilon R_1(t)})<\infty$. 
The same argument applies to reaction channel $5$.
To show that both $\mathbb{E}(e^{\epsilon R_2(t)})$ and $\mathbb{E}(e^{\epsilon R_6(t)})$ are finite (for some $\epsilon>0$), we couple the system with the following elementary system (with rate parameters $\alpha_1 + \beta_1$ and $\alpha_2 + \beta_2$)
\[
\widetilde{R}_1: \, \widetilde{S}_1 \xrightarrow{\alpha_1 + \beta_1} 2\widetilde{S}_1,
\quad 
\widetilde{R}_2: \, \widetilde{S}_2 \xrightarrow{\alpha_2 + \beta_2} 2\widetilde{S}_2
\]
through
\begin{equation*}
\begin{split}
X_1(t) = x_{0,1}  &+ Y_{1,1}\left(\int_0^t \alpha_1 \, ds\right) +
Y_{1,2}\left(\int_0^t \beta_1 (X_1(s) \wedge \widetilde{X}_1(s)) \,ds\right)\\
&+ Y_{1, 3}\left(\int_0^t \beta_1 (X_1(s) - X_1(s) \wedge \widetilde{X}_1(s)) \, ds \right)\\
&- Y_{1, 4}\left(\int_0^t \gamma_{12} X_1(s) X_2(s) \, ds\right) - Y_{1, 5}\left(\int_0^t \delta_1 X_1(s)\, ds\right),\\
X_2(t) = x_{0,2}  &+ Y_{2,1}\left(\int_0^t \alpha_2 \, ds\right) +
Y_{2,2}\left(\int_0^t \beta_2 (X_2(s) \wedge \widetilde{X}_2(s)) \,ds\right)\\
&+ Y_{2, 3}\left(\int_0^t \beta_2 (X_2(s) - X_2(s) \wedge \widetilde{X}_2(s)) \, ds \right)\\
&- Y_{2, 4}\left(\int_0^t \gamma_{21} X_2(s) X_1(s) \, ds\right) - Y_{2, 5}\left(\int_0^t \delta_2 X_2(s)\, ds\right)
\end{split}
\end{equation*}
and
\begin{equation*}
\begin{split}
\widetilde{X}_1(t) = \widetilde{x}_{0,1} &+ Y_{1, 1}\left(\int_0^t \alpha_1 \widetilde{X}_1(s)\, ds\right) + Y_{1, 2}\left(\int_0^t \beta_1 (X_1(s) \wedge
\widetilde{X}_1(s))\, ds\right)\\
&+ Y_{1, 6}\left(\int_0^t \beta_1 (\widetilde{X}_1(s) - X_1(s) \wedge \widetilde{X}_1(s)) \, ds\right), \\
\widetilde{X}_2(t) = \widetilde{x}_{0,2} &+ Y_{2, 1}\left(\int_0^t \alpha_2 \widetilde{X}_2(s)\, ds\right) + Y_{2, 2}\left(\int_0^t \beta_2 (X_2(s) \wedge
\widetilde{X}_2(s))\, ds\right)\\
&+ Y_{2, 6}\left(\int_0^t \beta_2 (\widetilde{X}_2(s) - X_2(s) \wedge \widetilde{X}_2(s)) \, ds\right) 
\end{split}
\end{equation*}
where $Y_{1, i}$ and $Y_{2, i}$ for $i = 1, \ldots, 6$ are independent unit rate Poisson processes
carried by $(\Omega, \mathcal{F}, \mathbb{P})$ and the initial population
$\widetilde{x}_{0,1} = x_{0,1}+ 1$ and $\widetilde{x}_{0,2} = x_{0,2}+ 1$. 
We show that $\mathbb{E}(e^{\epsilon R_2(t)})$ is finite for sufficiently small $\epsilon > 0$.
To this end, note that the reaction
event counts $R_1(t),R_2(t)$ and $\widetilde{R}_1(t)$ are given by
\begin{equation*}
\begin{split}
  R_1(t) =& Y_{1,1}\left(\int_0^t \alpha_1 \, ds\right),\\
  R_2(t) =& Y_{1,2}\left(\int_0^t \beta_1 (X_1(s) \wedge \widetilde{X}_1(s)) \,ds\right)
+ Y_{1,3}\left(\int_0^t \beta_1 (X_1(s) - X_1(s) \wedge \widetilde{X}_1(s)) \, ds
  \right),\\
  \widetilde{R}_1(t) =&  Y_{1,1}\left(\int_0^t \alpha_1 \widetilde{X}_1(s) \, ds\right) + Y_{1,2}\left(\int_0^t \beta_1 (X_1(s) \wedge
\widetilde{X}_1(s))\, ds\right)\\
&+ Y_{1,6}\left(\int_0^t \beta_1 (\widetilde{X}_1(s) -X_1(s) \wedge \widetilde{X}_1(s)) \, ds\right). 
\end{split}
\end{equation*}
We further note that $\widetilde{X}_1(t) \geq \widetilde{x}_{0,1} \geq 1$. Hence
\[
Y_{1,1}\left(\int_0^t \alpha_1 \, ds\right) \leq Y_{1,1}\left(\int_0^t \alpha_1 \widetilde{X}_1(s) \, ds\right).
\]
We show that $X_1(t) \leq \widetilde{X}_1(t)$ for all $t \geq 0$
following a similar argument as in Example~\ref{eg:coupling-1}. Define
the $n$th jump time of the combined process $(X_1(t),\widetilde{X}_1(t))$
to be $T_n$ and set $T_0=0$. We observe $X_1(T_0) \leq \widetilde{X}_1(T_0)$.
Suppose $X_1(T_j) \leq \widetilde{X}_1(T_j)$ for $j=0,\dots,n-1$ for some
$n$. Then, following an argument similar to that in
Example~\ref{eg:coupling-1}, (and noting $Y_{1, 3}(0)=0$) we obtain 
\[
\begin{aligned}
X_1(T_n) = x_{0,1}  &+ Y_{1,1}\left(\int_0^{T_n} \alpha_1 \, ds\right) +
Y_{1,2}\left(\int_0^{T_n} \beta_1 X_1(s) \,ds\right)\\
&- Y_{1,4}\left(\int_0^{T_n} \gamma_{12} X_1(s) X_2(s) \, ds\right) - Y_{1,5}\left(\int_0^{T_n} \delta_1 X_1(s)\, ds\right)\\
\end{aligned}
\]
and
\[
\begin{aligned}
\widetilde{X}_1(T_n) &= \widetilde{x}_{0,1} + Y_{1,1}\left(\int_0^{T_n} \alpha_1 \widetilde{X}_1(s)\, ds\right) + Y_{1,2}\left(\int_0^{T_n} \beta_1 X_1(s)\, ds\right)\\
&+ Y_{1,6}\left(\int_0^t \beta_1 (\widetilde{X}_1(s) - X_1(s)) \, ds\right),
\end{aligned}
\]
which shows $X_1(T_n) \leq \widetilde{X}_1(T_n)$ completing the induction.

Since $X_1(t) \leq  \widetilde{X}_1(t)$ and $\widetilde{X}_1(t) \geq 1$ for all $t \geq 0$, we obtain that
\[
\begin{aligned}
R_1(t) + R_2(t) &= Y_{1,1}\left(\int_0^t \alpha_1 \, ds\right) + Y_{1,2}\left(\int_0^t
\beta_1 X_1(s)\,ds\right) \\
&\leq Y_{1,1}\left(\int_0^{t} \alpha_1 \widetilde{X}_1(s)\,
ds\right) + Y_{1,2}\left(\int_0^{t} \beta_1 X_1(s)\, ds\right) 
 \leq \widetilde{X}_1(t)
\end{aligned}
\]
for all $t \geq 0$.
Since $\widetilde{X}_1(t)$ follows the negative binomial distribution,
there exists $\epsilon>0$ such that
\begin{equation}\label{eqn:coupling-2-integrability}
\mathbb{E}\left( e^{\epsilon (R_1(t) + R_2(t))}  \right) \leq \mathbb{E}\left( e^{\epsilon \widetilde{X}_1(t)} \right) < \infty.
\end{equation}
Applying the same argument to $R_5(t), R_6(t)$ and $\widetilde{R}_2(t)$ leads to
\begin{equation}\label{eqn:coupling-3-integrability}
\mathbb{E}\left( e^{\epsilon (R_5(t) + R_6(t))}  \right) \leq \mathbb{E}\left( e^{\epsilon \widetilde{X}_2(t)} \right) < \infty
\end{equation}
since $\widetilde{X}_2(t)$ is negative binomially distributed as well.

Finally, for sensitivities with respect to $\gamma_{12} (\gamma_{21} )$ or $\delta_1 (\delta_2)$,
note that
\begin{equation*}
\begin{split}
X_1(t) &= x_{0,1} + R_1(t) + R_2(t) - R_3(t) - R_4(t) \geq 0,\\
X_2(t) &= x_{0,2} + R_5(t) + R_6(t) - R_7(t) - R_8(t) \geq 0,
\end{split}
\end{equation*}
and hence $R_3(t) + R_4(t)$ and $R_5(t) + R_6(t)$ can be bounded by $x_{0,1} + R_1(t) + R_2(t)$ and $x_{0,2} + R_5(t) + R_6(t)$, respectively.
Their exponential integrability follows directly from~\eqref{eqn:coupling-2-integrability} and~\eqref{eqn:coupling-3-integrability}.
\end{example}

Examples~\ref{eg:coupling-1} and~\ref{eg:coupling-2} suggest a general
result for verifying Property~\ref{assume-unconsuming-reaction} for a class of reaction networks.
We summarize the result in the following theorem.
\begin{theorem}\label{thm:coupling}
Given a non-explosive chemical reaction network with product form propensity functions,
we consider a class of systems that satisfies the following conditions.
\begin{enumerate}
\item  All unconsuming reactions have the form   
$A \to A + S_i$ where $S_i$ is a species and $A$ is a
  species (or empty) with bounded population (we refer to this as {\em type 1} reaction)
  or have the form $S_i \to 2S_i$ where $S_i$ is a species ( we refer to the
  this as {\em type 2} reaction).
\item Let ${S_1,S_2,..,S_k}$ be the set of species involved in type 2
  unconsuming reactions in the previous condition. Then none of the consuming reactions result in an increase of any of the species $S_i$ for $i=1,..,k$.
\end{enumerate}
Then Property~\ref{assume-unconsuming-reaction} is satisfied.
\end{theorem}

\begin{remark}
Some remarks are in order. 
\begin{itemize}
\item 
The coupling strategy can definitely be applied to other classes of system
given that we can identify an elementary subsystem which satisfies
Property~\ref{assume-unconsuming-reaction} in order to bound the growth
of reaction counts of the full system.
\item The above result does not apply to the example such as
\[ S_1 \to 2S_1, \quad S_1 + S_2 \to 2 S_1, \quad 2S_1 \to S_1 + S_2,\]
where the consuming reaction $S_1 + S_2 \to 2S_1$ increases the population of $S_1$.
In this case, coupling the system with $\widetilde{S}_1 \to 2\widetilde{S}_1$ does not work because we can not bound the population of $S_1$ by that of $\widetilde{S}_1$.
Verifying Property~\ref{assume-unconsuming-reaction} for this type of examples will be the focus of another work. 
\end{itemize}
\end{remark}

\bibliographystyle{tfs}
\bibliography{GT}

\appendix

%
\section{Differentiating Inside an Integral}
\noindent
\begin{theorem}\label{thm:Glynn}{\bf (Asmussen \& Glynn) \cite{Glynn-book})}
Suppose $G(c, \omega)$ is a random variable for each $c$ in some interval of the real line. 
Let $c_{\text{ref}}$ be a specific value of $c$. Suppose the following hold:
\begin{enumerate}
\item For a set of $\omega$ with probability one, $G(c, \omega)$ is differentiable with respect to $c$ at $c = c_{\text{ref}}$.

\item There exists an interval $(c_l, c_u)$ containing $c_{\text{ref}}$ (independent of $\omega$) on which $G(c, \omega)$ is Lipschitz (in $c$) for a set of $\omega$ with probability one, with constant $K$ which may depend on $\omega$. That is, for any $c_1, c_2$ in the interval $(c_l, c_u)$, the following holds:
\[|G(c_1, \omega) - G(c_2, \omega)| \leq K(\omega)|c_1 - c_2|.\]

\item $\mathbb{E}(K)$ is finite.

\item $\mathbb{E}(|G(c, \omega)|)$ is finite for all $c$ in $(c_l, c_u)$.
\end{enumerate}
Then the following holds:
\[\left.\frac{d}{dc}\right|_{c=c_{\text{ref}}} \mathbb{E}(G(c)) = \mathbb{E}\left(\left.\frac{d}{dc}\right|_{c=c_{\text{ref}}} G(c)\right).\]
\end{theorem}

\section{Probability mass function of the pure birth process $S \to 2 S$}\label{app:negative-binomial}
We consider reaction system consisting of single species $S$ and single
reaction channel $S \to 2 S$. The species count process $X(t)$
evolves on state space is $\posint$. Let the reaction propensity be given by
$cx$. That is given $X(t)=x$, the probability of one reaction event during
$(t,t+h]$ is exactly $cxh + o(h)$ as $h \to 0+$. We suppose that
  $X(0)=x$ with probability one. 

Define $p_n(t) = \text{Prob}(X(t)=n)$ for $n \in \posint$. The Kolmogorov's
forward equations are given by
\begin{equation}
\begin{aligned}  
  p'_n(t) &= - c n p_n(t) + c (n-1) p_{n-1}(t) \quad n \geq x + 1,\\
  p'_x(t) &= - c x p_x(t).
\end{aligned}
\end{equation}
We note that $p_n(t)=0$ for all $n < x$ and $t \geq 0$. Also the
initial conditions are $p_x(0)=1$ and $p_n(0)=0$ for $n \neq x$.

For $x \geq 1$, we claim that the solution is given by
\begin{equation} 
  p_{x+k}(t) = \frac{(x+k-1)!}{k! (x-1)!} q^x p^k \quad k \geq 0,
\end{equation}
where $q=e^{-ct}$ and $p=1-q$. We note that when $k=0$, this gives $p_x(t) =
e^{-cxt}$. Thus the probability mass function follows a {\em negative
  binomial} distribution.  

Clearly, the formula for $p_x(t)$ is correct. To verify the correctness
of $p_{x+k}(t)$ for $k \geq 1$, we
simply verify the forward equations. We first note that $q'=-cq$ and $p'=cq$.
We obtain that
\[
p'_{x+k}(t) =  - \frac{(x+k-1)!}{k! (x-1)!}\, c\, x\, q^x \, p^k + \frac{(x+k-1)!}{k!
  (x-1)!}\, c \, k \, q^{x+1} \, p^{k-1}.
\]
Hence
\[
\begin{aligned}
p'_{x+k} + c(x+k) p_{x+k} &= \frac{(x+k-1)!}{k! (x-1)!}\, ck \, q^x \, p^k +
\frac{(x+k-1)!}{k! (x-1)!}\, ck \, q^{x+1} \, p^{k-1}\\
&= ck \, \frac{(x+k-1)!}{k! (x-1)!}\, q^x \, p^{k-1} \, (q + p) = c (x+k-1)
p_{x+k-1},
\end{aligned}
\]
which verifies the forward equations.

We show that for each $t >0$, there exists $\epsilon>0$ such that
$\bE(e^{\epsilon X(t)}) < \infty$. In fact
\[
\bE(e^{\epsilon X(t)}) = \sum_{k=0}^\infty e^{\epsilon (x+k)} p_{x+k}(t).
\]
To show that the sum is finite, we use the ratio test. The ratio of the
$k+1$st term to the $k$-th term
\[
M_{k+1}/M_k = e^{\epsilon x} \frac{x+k}{k+1} (1-e^{-ct})
\]
which limits to $e^{\epsilon x} (1-e^{-ct})$ as $k \to \infty$ and the limit is
less than $1$ for sufficiently small $\epsilon>0$.

\end{document}